\documentclass[11pt]{article}
\usepackage{amscd}
\usepackage{amsfonts}
\usepackage{amsmath}
\usepackage{amssymb}
\usepackage{amsthm}
\usepackage{bbm}
\usepackage{CJK}
\usepackage{fancyhdr}
\usepackage{graphicx}
\usepackage{indentfirst}
\usepackage{latexsym,bm}
\usepackage{mathrsfs}
\usepackage[arrow,matrix]{xy}
\usepackage[square, comma, sort&compress, numbers]{natbib}
\usepackage[colorlinks,linkcolor=red, anchorcolor=blue]{hyperref}
\allowdisplaybreaks[4]
\newtheorem{theorem}{Theorem}[section]
\newtheorem{lemma}[theorem]{Lemma}
\newtheorem{definition}[theorem]{Definition}
\newtheorem{proposition}[theorem]{Proposition}
\newtheorem{example}[theorem]{Example}

\newtheorem{corollary}[theorem]{Corollary}
\newtheorem{remark}[theorem]{Remark}

\usepackage[top=1in,bottom=1in,left=1.25in,right=1.25in]{geometry}
\def\<{\langle}
\def\>{\rangle}
\def\a{\alpha}
\def\b{\beta}

\def\B{\Box}
\def\c{\cdot}

\def\d{\delta}
\def\D{\Delta}

\def\e{\eta}

\def\lr{\longrightarrow}

\def\m{\mapsto}
\def\o{\otimes}
\def\om{\omega}

\def\r{\rho}

\def\ra{\rightarrow}

\def\si{\sigma}

\def\tr{\triangleright}
\def\tl{\triangleleft}

\def\v{\varepsilon}
\def\vp{\varphi}

\def\z{\zeta}

\date{}
\begin{document}
\renewcommand{\baselinestretch}{1.2}
\renewcommand{\arraystretch}{1.0}
\title{\bf BiHom-four-angle Hopf modules and BiHom-Yetter-Drinfel'd modules }
 \date{}
\author {{\bf Dongdong Yan$^{a}$  \quad Xiaoqian Liu $^{b}$\footnote {Corresponding author:  2021032@njxzc.edu.cn} }\\
{\small a: School of Mathematics and Physics, Nanjing Institute of Technology, Nanjing, }\\
{\small  Jiangsu 211167, P. R. of China}\\
{\small b: School of Information Engineering, Nanjing Xiaozhuang University, Nanjing }\\
{\small  Jiangsu 211171, P. R. of China}
} 
 \maketitle
\begin{center}
\begin{minipage}{12.cm}
\noindent{\bf Abstract.}
In this paper, we introduce the notion of four‑angle Hopf modules over a BiHom‑Hopf algebra $H$. We show that the category ${}_{H}^{H}\mathfrak{M}_{H}^{H}$ of BiHom‑four‑angle Hopf modules admits a strict monoidal category structure with respect to either the BiHom‑tensor product $\otimes_{H}$ or the BiHom‑cotensor product $\square_{H}$ as its monoidal product. We prove that the category $\mathcal{YD}_{H}^{H}(m,n,p,q)$ of BiHom‑$(m,n,p,q)$-Yetter‑Drinfel’d modules with paremeters $m,n,p,q\in\mathbb{Z}$ forms a strict braided monoidal category equipped with a new monoidal product. Furthermore, we establish monoidal equivalences between the monoidal categories $\mathcal{YD}_{H}^{H}(m,n,p,q)$ and ${}_{H}^{H}\mathfrak{M}_{H}^{H}$, where ${}_{H}^{H}\mathfrak{M}_{H}^{H}$ carries either $\otimes_{H}$ or $\square_{H}$ as its monoidal product. Finally, we construct braiding structures for the monoidal categories $\bigl({}_{H}^{H}\mathfrak{M}_{H}^{H},\otimes_{H}\bigr)$ and $\bigl({}_{H}^{H}\mathfrak{M}_{H}^{H},\square_{H}\bigr)$.
   \\

\\

 \end{minipage}
 \end{center}
 \normalsize\vskip1cm

\section*{Introduction}
Hopf modules constitute an important class of objects within Hopf algebra theory. A (right) Hopf module over a Hopf algebra $H$ is simultaneously a right $H$-module and a right $H$-comodule subject to a compatibility condition, which differs substantially from the compatibility condition defining a right‑right Yetter-Drinfel'd module. Its intrinsic nature can be best understood via the fundamental theorem of Hopf modules (see \cite{S69}). In \cite{W89}, Woronowicz redeveloped Hopf modules, together with their fundamental theorem, to investigate differential calculi over quantum groups. Schauenburg \cite{S94} further generalized this fundamental theorem to establish a monoidal equivalence between the category of two‑sided two‑cosided Hopf modules over $H$ and the category of Yetter-Drinfel'd modules over $H$. Moreover, this equivalence is monoidal when the category of Yetter-Drinfel'd modules is equipped with the tensor product over $\Bbbk$, while the category of two‑sided two‑cosided Hopf modules is endowed with either the tensor product or the cotensor product over $H$ (see \cite{D81}). For related generalizations, we refer the reader to \cite{ZWW19,GW20,H25,LYDW26}.

Hom‑algebras were first introduced by Makhlouf and Silvestrov in \cite{MS08}. Within this framework, ordinary associativity is replaced by Hom‑associativity, given by the relation $\alpha(a)bc=(ab)\alpha(c)$. Hom‑coassociativity for Hom‑coalgebras can be defined in an analogous fashion (see \cite{MS10}). The notions of Hom‑bialgebras and Hom‑Hopf algebras have likewise been formulated and systematically developed; see \cite{FK20, G10, LS14, MLC20,  MP14, MP15}. From the viewpoint of monoidal category theory, Caenepeel and Goyvaerts studied Hom‑structures in \cite{CG11}, where they introduced monoidal Hom‑algebras, monoidal Hom‑coalgebras, and related objects inside symmetric monoidal categories. These concepts differ slightly from the Hom‑algebras and Hom‑coalgebras mentioned above.

BiHom‑(co)algebras and BiHom‑bialgebras were investigated by Graziani et al. in \cite{GMMP15}, yielding a more general theoretical setting. More precisely, a BiHom‑bialgebra is a generalized bialgebra whose associativity and unit conditions are twisted by two automorphisms $\alpha$ and $\beta$, whereas its coassociativity and counit conditions are twisted by another pair of automorphisms $\varphi$ and $\psi$. It recovers the usual Hom‑bialgebra when $\alpha=\beta=\varphi=\psi$, and reduces to a monoidal Hom‑bialgebra under the condition $\alpha^{-1}=\beta^{-1}=\varphi=\psi$. Further investigations concerning BiHom‑type algebras may be found in \cite{FL18,GZW18,LMMP20,ZW20, ZWC24} and other references.

A natural question to ask is whether the main theorem in \cite{S94} still holds in the setting of BiHom-Hopf algebras.

This paper is organized as follows. In Section 1, we recall the definitions and properties of BiHom‑type structures as well as braided monoidal categories. Let $H$ be a BiHom‑Hopf algebra with bijective antipode. In Section 2, we introduce the definition of four‑angle Hopf modules over $H$, and prove that the category ${}_{H}^{H}\mathfrak{M}_{H}^{H}$ of BiHom‑four‑angle Hopf modules over $H$ admits two strict monoidal category structures. The monoidal product for objects $M,N\in{}_{H}^{H}\mathfrak{M}_{H}^{H}$ are given by $M\otimes_{H}N$ and $M\square_{H}N$, respectively. In Section 3, we review the notion of BiHom‑$(m,n,p,q)$-Yetter‑Drinfel’d modules over $H$ (see \cite{YL26}), where $m,n,p,q\in\mathbb{Z}$,  and verify that the category $\mathcal{YD}_{H}^{H}(m,n,p,q)$ of BiHom‑$(m,n,p,q)$-Yetter‑Drinfel’d modules is a strict braided monoidal category with a redefined monoidal product. In Section 4, we first establish a monoidal category equivalence between the monoidal category $\mathcal{YD}_{H}^{H}(m,n,p,q)$ and ${}_{H}^{H}\mathfrak{M}_{H}^{H}$, where the latter category is endowed with either $\otimes_{H}$ or $\square_{H}$ as its monoidal product. This result generalizes the main theorem in \cite{S94}. Finally, we construct the braiding structure for the monoidal categories $\bigl({}_{H}^{H}\mathfrak{M}_{H}^{H},\otimes_{H}\bigr)$ and $\bigl({}_{H}^{H}\mathfrak{M}_{H}^{H},\square_{H}\bigr)$.

Throughout this paper, all algebraic systems are over a field $\Bbbk$.  We denote the identity map by $\mathrm{id}$. We shall use the sigma notation in the versions of Sweedler
for $\D:\D(h)=h_{1}\o h_{2}$. In order to facilitate our computations, we always omit the summation symbol $\sum$.

\section{preliminaries}
\def\theequation{1.\arabic{equation}}
\setcounter{equation} {0}

A \emph{monoidal category} $\mathcal{C}=(\mathcal{C},\otimes,\mathcal{I},a,l,r)$ is a category $\mathcal{C}$ equipped with a tensor product functor $\otimes:\mathcal{C}\times\mathcal{C}\longrightarrow\mathcal{C}$, with a tensor unit object $\mathcal{I}\in\mathcal{C}$, with an associativity constraint isomorphism $a=a_{U,V,W}:(U\otimes V)\otimes W\longrightarrow U\otimes(V\otimes W)$ for any objects $U,V,W\in\mathcal{C}$, a left unit constraint $l=l_U:\mathcal{I}\otimes U\longrightarrow U$ and a right unit constraint $r=r_U:U\otimes\mathcal{I}\longrightarrow U$, for any object $U\in\mathcal{C}$, such that the pentagon axiom $a_{U,V,W\otimes X}\circ a_{U\otimes V,W,X}=(U\otimes a_{V,W,X})\circ a_{U,V\otimes W,X}\circ(a_{U,V,W}\otimes X)$ and the triangle axiom $(U\otimes l_V)\circ a_{U,\mathcal{I},V}=(r_U\otimes V)$ hold, for any objects $U,V,W,X\in\mathcal{C}$. A monoidal category $\mathcal{C}$ is \emph{strict} if all the constraints are identities. 

A \emph{braiding}\cite{K95} of a monoidal category $\mathcal{C}$ is a family of natural isomorphisms $c=c_{V,W}:V\otimes W\longrightarrow W\otimes V$ such that the following conditions hold
\begin{align*}
\begin{cases}
&c_{U,V\otimes W}=a_{V,W,U}^{-1}\circ(V\otimes c_{U,W})\circ a_{V,U,W}\circ(c_{U,V}\otimes W)\circ a_{U,V,W}^{-1},\\
&c_{U\otimes V,W}=a_{W,U,V}\circ(c_{U,W}\otimes V)\circ a_{U,W,V}^{-1}\circ(U\otimes c_{V,W})\circ a_{U,V,W},
\end{cases}
\end{align*}
for any $U,V,W\in\mathcal{C}$, where $a$ is the associativity constraint in the monoidal category $\mathcal{C}$.

Note that a braided monoidal category is a monoidal category $\mathcal{C}$ with a braiding.

In what follows, we will recall from \cite{GMMP15, ZWC24}  some information concerning BiHom-structures.

A \emph{unital BiHom-associative algebra $A$} is a 5-tuple $(A, \mu_A, 1_A, \alpha_A, \beta_A)$, in which $A$ is a linear space, $1_A \in A$ is an element (the unit), $\alpha_A, \beta_A: A \to A$ are linear isomorphisms, $\mu_A: A \otimes A \to A$ is a linear map with the notation $\mu_{A}(a \otimes b) = ab$, such that, for all $a,b,c \in A$:
\begin{align*}
\begin{cases}
&\alpha_A(1_A) = \beta_A(1_A) = 1_A,\quad a1_A = \alpha_A(a),\quad 1_Aa = \beta_A(a),\\
& \alpha_A(a)(bc) = (ab)\beta_A(c),\\
&\alpha_A \circ \beta_A = \beta_A \circ \alpha_A,\quad \alpha_A(ab) = \alpha_A(a)\alpha_A(b),\quad \beta_A(ab) = \beta_A(a)\beta_A(b).
\end{cases}
\end{align*}

In this paper, the algebras we mainly discussed are this kind of unital BiHom-associative algebras, and in the following we call them the \emph{BiHom-algebras}.

A $\emph{BiHom-algebra map}$ $f:(A, \mu_A, 1_A, \alpha_A, \beta_A)\lr (A', \mu_{A'}, 1_{A'}, \alpha_{A'}, \beta_{A'})$ is a map $f:A\lr A'$ such that $\a_{A'}\circ f=f\circ \a_{A}$, $\b_{A'}\circ f=f\circ \b_{A}$, $f(ab)=f(a)f(b)$ and $f(1_{A})=1_{A'}$,  for any $a,b\in A$.

A  \emph{counital BiHom-coassociative coalgebra $C$} is a 5-tuple $(C, \Delta_C, \varepsilon_C, \om_C, \psi_C)$, in which $C$ is a linear space, $\om_C, \psi_C: C \to C$ are linear isomorphisms, $\varepsilon_C: C \to \Bbbk$ and $\Delta_C: C \to C \otimes C$ are linear maps, such that, for all $c\in C$:
\begin{align*}
\begin{cases}
&\varepsilon_C(\om_C(c)) = \varepsilon_C(\psi_C(c)) = \varepsilon_C(c),\quad c_1\varepsilon_C(c_2) = \om_C(c),\quad \varepsilon_C(c_1)c_2 = \psi_C(c),\\
& \om_C(c_1) \otimes \Delta_C(c_2) = \Delta_C(c_1) \otimes \psi_C(c_2),\\
&\om_C \circ \psi_C = \psi_C \circ \om_C,\quad \Delta_C(\om_C(c)) = \om_C(c_1) \otimes \om_C(c_2),\quad \Delta_C(\psi_C(c)) = \psi_C(c_1) \otimes \psi_C(c_2).
\end{cases}
\end{align*}

Analogue to BiHom-algebras, \emph{BiHom-coalgebras} will be short for counital BiHom-coassociative coalgebra without any confusion.

A \emph{BiHom-coalgebras map $f:(C, \Delta_C, \varepsilon_C, \om_C, \psi_C)\lr (C', \Delta_{C'}, \varepsilon_{C'}, \om_{C'}, \psi_{C'})$} is a map $f:C\lr C'$ such that $\om_{C'}\circ f=f\circ \om_{C}$, $\psi_{C'}\circ f=f\circ \psi_{C}$, $\D_{C'}\circ f=(f\o f)\circ \D_{C}$ and $\varepsilon_{C'}\circ f=\varepsilon_{C}$.

A \emph{BiHom-bialgebra $H$} is a 9-tuple $(H,\mu,1,\Delta,\varepsilon,\a,\b,\om,\psi)$ with the property that $(H,\mu,1,\a,\b)$ is a BiHom-algebra, $(H,\Delta,\varepsilon,\om,\psi)$ is a BiHom-coalgebra, and $\Delta,\varepsilon$ are all morphisms of BiHom-algebras preserving unit, i.e., for all $h,g\in H$,
$$
\Delta(hg)=h_1g_1\otimes h_2g_2,\quad \v(hg)=\v(h)\v(g),\quad \Delta(1)=1\otimes 1,\quad \v(1)=1_\B.
$$
Moreover, $\a,\b$ are BiHom-coalgebra maps, $\om,\psi$ are BiHom-algebra maps, and they commute with each other.

A \emph{BiHom-Hopf algebras} \cite{ZWC24}  is a BiHom-bialgebra $H:=(H,\mu,1,\Delta,\varepsilon,\a,\b,\om,\psi)$ with a morphism (called the antipode)
$S\colon H\to H$ such that $S$ commutes with $\alpha,\beta,\om,\psi$, and satisfies, for any $h\in H$,
$$
h_1S(h_2)=S(h_1)h_2=\varepsilon(h)1_H.
$$

Note the definition of BiHom-Hopf algebras which is little different from \cite{GMMP15}, Def. 6.9.

\begin{proposition}\cite{ZWC24}
If $H$ is a BiHom-Hopf algebra, then

(1) the antipode $S$ satisfies 
\begin{align*}
&S(ab) = S\alpha^{-1}\beta(b)\,S\alpha\beta^{-1}(a),\quad S(1)=1,
\\&\Delta(S(a)) = S\om\psi^{-1}(a_2)\otimes S\om^{-1}\psi(a_1),\quad \varepsilon\circ S=\varepsilon,
\\&S\alpha^2\om^2=S\beta^2\psi^2;
\end{align*}

(2) if $S$ is a bijective map, then 
\begin{align*}
&\alpha^2\om^2=\beta^2\psi^2,\\
&S^{-1}(ab) = S^{-1}\alpha^{-1}\beta(b)S^{-1}\alpha\beta^{-1}(a),\quad S^{-1}(1_H)=1_H,\\
&\Delta(S^{-1}(a)) = S^{-1}\om\psi^{-1}(a_2)\otimes S^{-1}\om^{-1}\psi(a_1),\quad \varepsilon\circ S^{-1}=\varepsilon,\\
&S^{-1}\alpha^{-2}\beta^2(a_2)a_1 = a_2S^{-1}\alpha^2\beta^{-2}(a_1)=\varepsilon(a)1_H,\\
&S^{-1}\om^{2}\psi^{-2}(a_2)a_1 = a_2S^{-1}\om^{-2}\psi^2(a_1)=\varepsilon(a)1_H.
\end{align*}

\end{proposition}

Let $H$ be a BiHom-bialgebra. A \emph{left $H$-module} is a 6-tuple $(M,\tr_M,\alpha_M,\beta_M,\om_M,\psi_M)$, in which $M$ is a linear space, $\alpha_M,\beta_M,\om_M,\psi_M\colon M\to M$ are linear isomorphisms, $\tr \colon H\otimes M\to M : h\otimes m\m h\tr m$ is linear map, such that, for any $h,g\in H$, $m\in M$,
\begin{align*}
\begin{cases}
&\alpha_M,\beta_M,\om_M,\psi_M \text{ commute with each other},\\
&\alpha(h)\tr \alpha_M(m)=\alpha_M(h\tr m),\quad \beta(h)\tr \beta_M(m)=\beta_M(h\tr m),\\
&\om(h)\tr \om_M(m)=\om_M(h\tr m),\quad \psi(h)\tr \psi_M(m)=\psi_M(h\tr m),\\
&\alpha(h)\tr (g\tr m)=(hg)\tr \beta_M(m),\quad 1_H\tr m=\beta_M(m).
\end{cases}
\end{align*}

Similarly, we can define the \emph{right $H$-module}.

A \emph{left $H$-module map} $f:(M,\tr_M,\alpha_M,\beta_M,\om_M,\psi_M)\lr (N,\tr_N,\alpha_N,\beta_N,\om_N,\psi_N)$ is a map $f:M\lr N$ such that $\alpha_N\circ f=f\circ \alpha_M, \beta_N\circ f=f\circ \beta_M, \om_N\circ f=f\circ \om_M, \psi_N\circ f=f\circ \psi_M$ 
and $f\circ \tr_M=\tr_N\circ (\mathrm{id}_H\o  f)$. Moreover, if $(M,\tr_M,\alpha_M,\beta_M,\om_M,\psi_M)$ is both a left $H$-module (via action $\tr$) and right $H$-module (via action $\tl$) and satisfies
\begin{align*}
\a_{H}(a)\tr (m\tl b)=(a\tr m)\tl \b_{H}(b),
\end{align*}
for all $a,b\in H, m\in M$, we call $M$ is an \emph{$H$-bimodule}.

Let $H$ be a BiHom-bialgebra. A \emph{right $H$-comodule} is a 6-tuple $(M, \r_M, \alpha_M,\beta_M,\om_M,\psi_M)$, in which  $M$ is a linear space, $\alpha_M,\beta_M,\om_M,\psi_M\colon M\to M$ are linear isomorphisms, $\r\colon M\to M\otimes H : m \m m_{(0)}\otimes m_{(1)}$ is linear map, such that,  for any $m\in M$,
\begin{align*}
\begin{cases}
&\om_M,\psi_M,\alpha_M,\beta_M \text{ commute with each other},\\
&(\alpha_M\otimes\alpha_{H})\circ\r=\r\circ\alpha_M,\quad (\beta_M\otimes\beta_{H})\circ\r=\r\circ\beta_M,\\
&(\om_M\otimes\om_{H})\circ\r=\r\circ\om_M,\quad (\psi_M\otimes\psi_{H})\circ\r=\r\circ\psi_M,\\
&\om_M(m_{(0)})\otimes m_{{(1)1}}\otimes m_{{(1)2}}=m_{(0)(0)}\otimes m_{(0)(1)}\otimes\psi_{H}(m_{(1)}),\quad m_{(0)}\varepsilon(m_{(1)})=\om_M(m).
\end{cases}
\end{align*}

Similarly, we can define the \emph{left $H$-comodule}.

A \emph{right $H$-comodule map} $f:(M,\r_M,\alpha_M,\beta_M,\om_M,\psi_M)\lr (N,\r_N,\alpha_N,\beta_N,\om_N,\psi_N)$ is a map $f:M\lr N$ such that 
$\alpha_N\circ f=f\circ\alpha_M, \beta_N\circ f=f\circ\beta_M,\om_N\circ f=f\circ\om_M,
\psi_N\circ f=f\circ\psi_M, \r_N\circ f=(f\otimes\mathrm{id}_H)\circ\r_M$.
Moreover, if $(M,\r_M,\alpha_M,\beta_M,\om_M,\psi_M)$ is both a left $H$-comodule (via coaction $\r^{l}_M(m)=m_{[-1]}\o m_{[0]}$) and right $H$-comodule (via action $\r^{r}_M(m)=m_{(0)}\o m_{(1)}$) and satisfies
\begin{align*}
&\om_{H}(m_{[-1]})\o m_{[0](0)}\o m_{[0](1)}=m_{(0)[-1]}\o m_{(0)[0]}\o \psi_{H}(m_{(1)}),
\end{align*}
for all $m\in M$, we call $M$ is a \emph{$H$-bicomodule}.

\section{BiHom-four-angle Hopf modules}
\def\theequation{2.\arabic{equation}}
\setcounter{equation} {0}
In this section, we introduce the concept of BiHom-four-angle Hopf modules and equip the category of BiHom-four-angle Hopf modules with two monoidal structure.
\begin{definition}
Let $H$ be a BiHom-bialgebra , $M$ a linear space and $\a_{M},\b_{M},\om_{M},\psi_{M}\in Aut(M)$. Then $M=(M,\a_{M},\b_{M},\om_{M},\psi_{M})$ is called a left-left BiHom-Hopf module if
\begin{enumerate}
\item[$(i)$] $(M, \c)$ is a left $H$-module;
\item[$(ii)$] $(M, \r)$ is a left $H$-comodule;
\item[$(iii)$] the following compatibility condition holds
\begin{align}\label{e2.1}
(h\c m)_{[-1]}\o (h\c m)_{[0]}=h_{1} m_{[-1]}\o h_{2}\c m_{[0]},
\end{align}
\end{enumerate}
for any $h\in H$ and $m\in M$.
\end{definition}
As above, we can also define the left-right, right-left, and right-right BiHom-Hopf modules as follows.
\begin{definition}\label{D2.3}
Let $H$ be a BiHom-bialgebra, $M$ a linear space and $\a_{M},\b_{M},\om_{M},\psi_{M}\in Aut(M)$. Then
\begin{enumerate}
 \item[$(1)$] $M$ is called a left-right BiHom-Hopf module if $M$ is both a left $H$-module and a right $H$-comodule such that the compatibility condition holds:
    \begin{align}\label{e2.2}
(h\c m)_{(0)}\o (h\c m)_{(1)}=h_{1}\c m_{(0)}\o h_{2}m_{(1)},
    \end{align}
    for any $h\in H$ and $m\in M$.
\item[$(2)$] $M$ is called a right-left BiHom-Hopf module if $M$ is both a right $H$-module and a left $H$-comodule such that the compatibility condition holds:
    \begin{align}\label{e2.3}
(m\c h)_{[-1]}\o (m\c h)_{[0]}= m_{[-1]}h_{1}\o m_{[0]}\c h_{2} ,
    \end{align}
    for any $h\in H$ and $m\in M$.
\item[$(3)$] $M$ is called a right-right BiHom-Hopf module if $M$ is both a right $H$-module and a right $H$-comodule such that the compatibility condition holds:
    \begin{align}\label{e2.4}
  (m\c h)_{(0)}\o (m\c h)_{(1)}=m_{(0)}\c h_{1}\o m_{(1)}h_{2},
    \end{align}
    for any $h\in H$ and $m\in M$.

\end{enumerate}
\end{definition}

Let $H$ be a BiHom-bialgebra. We denote the categories of BiHom-Hopf modules by $\!^{H}_{H}\mathfrak{M}$, $\!_{H}\mathfrak{M}^{H}$, $\!^{H}\mathfrak{M}_{H}$ and $\mathfrak{M}^{H}_{H}$. Let $\mathfrak{M}^{H}_{H}$ be the category whose objects are all right-right BiHom-Hopf modules over $H$; the morphisms in the category are morphisms of right $H$-modules and right $H$-comodules.
\begin{example}\label{E1}
Let $H$ be a BiHom-bialgebra, $V$ a linear space and $\a_{V},\b_{V},\om_{V},\psi_{V}\in Aut(V)$. Then 
\begin{enumerate}
\item[$(1)$] $H\o V:=(H\o V, \a\o \a_{V}, \b\o \b_{V}, \om\o \om_{V}, \psi\o \psi_{V})$ is a left-left BiHom-Hopf module with the following structures:
\begin{align*}
\begin{cases}
h\c(g\o v)=hg\o \b_{V}(v)
\\\r^{l}_{H\o V}(g\o v)=g_{1}\o g_{2}\o \psi_{V}(v)
\end{cases}
\end{align*}
for any $h,g\in H$ and $v\in V$.
\item[$(2)$] $H\o H:=(H\o H, \a\o \a, \b\o \b, \om\o \om, \psi\o \psi)$ is a left-left BiHom-Hopf module with the following structure:
    \begin{align*}
\begin{cases}
h\c(g\o k)=hg\o \b(k)
\\\r^{l}_{H\o H}(g\o k)=g_{1}\o g_{2}\o \psi(k)
\end{cases}
\end{align*}
for any $h,g,k\in H$.
\end{enumerate}
\end{example}
\begin{definition}\label{L2.4}
Let $M$ be a linear space and $\a_{M},\b_{M},\om_{M},\psi_{M}\in Aut(M)$. Then the following statements are equivalent:
\begin{enumerate}
\item[$(i)$] $M$ is an $H$-bimodule in $\!^{H}\mathfrak{M}$.
\item[$(ii)$] $M$ is a left $H$-comodule in $\!_{H}\mathfrak{M}_{H}$.
\item[$(iii)$] $M$ is an $H$-bimodule and a left $H$-comodule such that $M\in \!^{H}\mathfrak{M}_{H}$ and $M\in \!_{H}^{H}\mathfrak{M}$
\end{enumerate}
\end{definition}
We call $M$ a two-sided BiHom-Hopf module and denote by $\!_{H}^{H}\mathfrak{M}_{H}$ the category of these objects with morphisms are left and right linear and left colinear. In the same manner, we can define the category $\!_{H}\mathfrak{M}^{H}_{H}$.

Similarly, we can define the category $\!_{H}^{H}\mathfrak{M}^{H}$ (and $\!^{H}\mathfrak{M}^{H}_{H}$), whose objects are called two-cosided 
BiHom-Hopf modules.
\begin{definition}\label{L2.5}
Let $M$ be a linear space and $\a_{M},\b_{M},\om_{M},\psi_{M}\in Aut(M)$. Then the following statements are equivalent:
\begin{enumerate}
\item[$(i)$] $M$ is an $H$-bimodule in $\!^{H}\mathfrak{M}^{H}$.
\item[$(ii)$] $M$ is an $H$-bicomodule in $\!_{H}\mathfrak{M}_{H}$.
\item[$(iii)$] $M$ is an $H$-bimodule and an $H$-bicomodule such that $M\in \!_{H}^{H}\mathfrak{M}, \!_{H}\mathfrak{M}^{H},$ $\!^{H}\mathfrak{M}_{H}, \mathfrak{M}_{H}^{H}$.
\end{enumerate}
\end{definition}
We call $M$ a \emph{BiHom-four-angle Hopf module} and denote by $\!_{H}^{H}\mathfrak{M}_{H}^{H}$ the category of these objects  with morphisms are linear and colinear on both sides. Also, There are some examples of BiHom-four-angle Hopf module.
\begin{example}\label{E2.7}
Let $H$ be a BiHom-bialgebra.
\begin{enumerate}
\item[$(1)$] $H_{a}:=(H\o H, \a\o \a, \b\o \b, \om\o \om, \psi\o \psi)$ is an object in $\!_{H}^{H}\mathfrak{M}_{H}^{H}$ with the following structures:
\begin{align*}
\begin{cases}
h\c (x\o y)=\a^{-1}(h)x\o \b(y),\quad \r^{l}_{a}(x\o y)=x_{1}y_{1}\o x_{2}\o y_{2},
\\(x\o y)\c h=\a(x)\o y\b^{-1}(h),\quad \r^{r}_{a}(x\o y)=x_{1}\o y_{1}\o x_{2}y_{2},
\end{cases}
\end{align*}
for any $x, y, h\in H$.
\item[$(2)$] $H_{b}:=(H\o H, \a\o \a, \b\o \b, \om\o \om, \psi\o \psi)$ is an object in $\!_{H}^{H}\mathfrak{M}_{H}^{H}$ with the following structures:
\begin{align*}
\begin{cases}
h\c (x\o y)=h_{1}x\o h_{2}y,\quad \r^{l}_{b}(x\o y)=\om^{-1}(x_{1})\o x_{2}\o \psi(y),
\\(x\o y)\c h=xh_{1}\o yh_{2},\quad\r^{r}_{b}(x\o y)=\om(x)\o y_{1}\o \psi^{-1}(y_{2}),
\end{cases}
\end{align*}
for any $x, y, h\in H$.
\end{enumerate}
\end{example}

In what follows, we give two monoidal structures on the category of BiHom-four-angle Hopf modules. First, we introduce the first structure. Let $M,N\in \!_{H}^{H}\mathfrak{M}_{H}^{H}$. The BiHom-tensor product $M\o_{H}N$ of $M$ and $N$ is defined by
\begin{align*}
M\o_{H}N:=\{m\o n\in M\o N ~| ~m\c h\o \b_{N}(n)=\a_{M}(m)\o h\c n, \forall ~h\in H\}.
\end{align*}

\begin{proposition}\label{P2.9}
Let $H$ be a BiHom-bialgebra . Then $(\!^{H}_{H}\mathfrak{M}^{H}_{H},$ $\o_{H})$ is a strict monoidal category with the structures defined as follows. For any $h\in H$ and $m\o n\in M\o_{H}N$, 
\begin{align*}
h\c(m\o n)&=T^{-1}\a^{-1}(h)\c m\o \b_{N} (n),
\\ \r^{l}(m\o n)&= T(m_{[-1]}n_{[-1]})\o m_{[0]}\o n_{[0]},
\\(m\o n)\c h&=\a_{M} (m)\o n\c T^{-1}\b^{-1}(h),
\\ \r^{r}(m\o n)&=m_{(0)}\o n_{(0)}\o T(m_{(1)}n_{(1)}),
\end{align*}
where $T=\a^{a}\b^{b}\om^{c}\psi^{d}$ with $a,b,c,d\in\mathbb{Z}$.
\end{proposition}
\begin{proof}
First, we check that the actions stated are well defined. For any $h,g\in H$ and $m\o_{H} n\in M\o_{H}N$, we have
\begin{align*}
h\c(m\c g\o_{H} \b_{N}(n))&=T^{-1}\a^{-1}(h)\c (m\c g)\o_{H} \b_{N}^{2} (n)
\\&=(T^{-1}\a^{-2}(h)\c m)\c\b( g)\o_{H} \b_{N}^{2} (n)
\\&=T^{-1}\a^{-1}(h)\c \a_{M}(m)\o_{H}\b( g)\c \b_{N}  (n)
\\&=h\c(\a_{M}(m)\o_{H} g\c n)
\end{align*}
Thus $h\c(m\o_{H} n)\in M\o_{H}N$. Similarly, we have $(m\o_{H} n)\c h\in M\o_{H}N$.
It is easy to show that $M\o_{H}N$ is both an $H$-bimodule and an $H$-bicomodule.

Finally, we only verify the compatibility conditions Eq. $(\ref{e2.1})$, and the rest, Eqs. $(\ref{e2.2})$, $(\ref{e2.3})$ and $(\ref{e2.4})$, are left to reader. For any $h,g\in H$ and $m\o n\in M\o_{H}N$, we have
\begin{align*}
\r^{l}(h\c(m\o n))&=\r^{l}(T^{-1}\a^{-1}(h)\c m\o \b_{N} (n))
\\&=T((T^{-1}\a^{-1}(h)\c m)_{[-1]}\b(n_{[-1]}))\o (T^{-1}\a^{-1}(h)\c m)_{[0]}\o \b_{N}(n_{[0]})
\\&=T((T^{-1}\a^{-1}(h_{1})m _{[-1]})\b(n_{[-1]}))\o T^{-1}\a^{-1}(h_{2})\c m _{[0]}\o \b_{N}(n_{[0]})
\\&=h_{1}T(m_{[-1]}n_{[-1]})\o h_{2}\c [m_{[0]}\o n_{[0]}]
\\&=h_{1}(m\o n)_{[-1]}\o h_{2}\c (m\o n)_{[0]}.
\end{align*}

This completes the proof.
\end{proof}

Now, we describe the second monoidal structure on this category, which can be viewed as a dual of the first one. Let Let $M,N\in \!_{H}^{H}\mathfrak{M}_{H}^{H}$. The BiHom-cotensor product $M\B_{H}N$ of $M$ and $N$ is defined by
\begin{align}\label{e2.5}
M\B_{H}N:=\{m\o n\in M\o N ~| ~\r^{r}(m)\o \psi_{N}(n)=\om_{M}(m)\o \r^{l}(n)\}.
\end{align}

As a dual of the result of Proposition \ref{P2.9}, we can get the following consequence.
\begin{proposition}\label{P1}
Let $H$ be a BiHom-bialgebra . Then $(~\!^{H}_{H}\mathfrak{M}^{H}_{H},$ $\B_{H})$ is a strict monoidal category with the following structures. For any $m\in M$ and $h, g\in H$,
\begin{align*}
&h\c(m\B n)=\hat{T}(h_{1})\c m\o \hat{T}(h_{2})\c n,&& \r^{l}(m\B n)= \hat{T}^{-1}\om^{-1}(m_{[-1]})\o m_{[0]}\o \psi_{N}(n),
\\&(m\B n)\c h=m\c \hat{T}(h_{1})\o n\c \hat{T}(h_{2}),&& \r^{r}(m\B n)= \om_{M}(m)\o n_{(0)}\o \hat{T}^{-1}\psi^{-1}(n_{(1)}).
\end{align*}
where $\hat{T}=\a^{c}\b^{d}\om^{a}\psi^{b}$ with $a,b,c,d\in\mathbb{Z}$.
\end{proposition}

\section{BiHom-Yetter-Drinfel'd modules}
\def\theequation{3.\arabic{equation}}
\setcounter{equation} {0}
In this section, we first recall the definition of Yetter-Drinfel'd modules over a BiHom‑bialgebra $H$. Then, we equip the category of BiHom‑$(m,n,p,q)$‑Yetter-Drinfel'd modules with a strict braided monoidal structure, where $m,n,p,q\in\mathbb{Z}$. 
\begin{definition} (see \cite{YL26})
Let $H$ be a BiHom-bialgebra, $V$ a linear space and $\a_{V},\b_{V},\om_{V},\psi_{V}\in Aut(V)$. Then $V$ is called a right-right BiHom-$(m,n,p,q)$-Yetter-Drinfel'd module over $H$ if 
\begin{enumerate}
\item[$(i)$] $(V, \tl)$ is a right $H$-module;
\item[$(ii)$] $(V,\r)$ is a right $H$-comodule;
\item[$(iii)$] the following compatibility condition holds
\begin{align}\label{e3.1}
[v\tl h_2]_{(0)}&\o \alpha^{m-2}\b^{n+1}\omega^{p-2}\psi^{q+1}(h_{1})[v\tl h_2]_{(1)}\notag 
\\&=v_{(0)}\tl \psi(h_{1})\o\b(v_{(1)})\alpha^{m-1}\b^{n}\omega^{p-1}\psi^{q}(h_{2}),
\end{align}
\end{enumerate}
for any $h\in H$, $v\in V$ and $m,n,p,q\in\mathbb{Z}$.
\end{definition}

Let $H$ be a BiHom-bialgebra. We denote by $\mathcal{YD}^{H}_{H}(m,n,p,q)$ the category whose objects are all right-right BiHom-$(m,n,p,q)$-Yetter-Drinfel'd modules over $H$; the morphisms in the category are morphisms of right $H$-modules and right $H$-comodules.

The categories ${}_H\mathcal{YD}^H(m,n,p,q)$, ${}^H\mathcal{YD}_H(m,n,p,q)$, and ${}^H_H\mathcal{YD}(m,n,p,q)$ of left–right, right–left, and left–left BiHom-$(m,n,p,q)$-Yetter–Drinfel'd modules, respectively, are defined by the compatibility conditions, for any $h\in H$ and $v\in V$,
\begin{align*}
(h_1\tr v)_{(0)}&\o (h_1\tr v)_{(1)}\alpha^{m+1}\b^{n-2}\omega^{p-2}\psi^{q+1}(h_{2})
\\&=\psi(h_{1})\tr v_{(0)}\o \alpha^{m}\b^{n-1}\omega^{p-1}\psi^{q}( h_{2})\a( v_{(1)}),
\\\alpha^{m-2}\b^{n+1}&\omega^{p+1}\psi^{q-2}(h_{1})[v\tl h_2]_{[-1]}\o[v\tl h_2]_{[0]}
\\&=\b(v_{[-1]})\alpha^{m-1}\b^{n}\omega^{p}\psi^{q-1}(h_{1})\o v_{[0]}\tl \om(h_{2}),
\\(h_1\tr v)_{[-1]}&\alpha^{m+1}\b^{n-2}\omega^{p+1}\psi^{q-2}(h_{2})\o (h_1\tr v)_{[0]}
\\&=\alpha^{m}\b^{n-1}\omega^{p}\psi^{q-1}(h_{1})\a(v_{[-1]})\o \om ( h_{2}) \tr v_{[0]}.
\end{align*}
\begin{example}\label{E3.3}
$(1)$
A BiHom-Hopf algebra $H$ with a bijective antipode $S$ can be considered as a right-right BiHom-$(m,n,p,q)$-Yetter-Drinfel'd module over itself with the comultiplication $\D$ as a right $H$-comodule, with the structure
\begin{align*}
x\leftharpoonup h=S \a^{m-3}\b^{n+1}\om^{p-3}\psi^{q}(h_{1}) [\b_H^{-1}(t)\a^{m-3}\b^{n}\om^{p-3}\psi^{q}(h_{2})],\quad \forall ~x, h\in H
\end{align*}
as a right $H$-module, and denote it by $H_{A}:=(H, \leftharpoonup, \D)$.

$(2)$
A BiHom-Hopf algebra $H$ with a bijective antipode $S$ can be considered as a right-right BiHom-$(m,n,p,q)$-Yetter-Drinfel'd module over itself with the multiplication $\mu$ as a right $H$-module, with the structure
\begin{align*}
\hat{\r}(h)=\psi^{-1}(h_{21})\o S\a^{m-3}\b^{n}\om^{p-3}\psi^{q+1}(h_{1})\a^{m-3}\b^{n}\om^{p-3}\psi^{q}(h_{22}),\quad \forall ~h\in H
\end{align*}
as a right $H$-comodule, and denote it by $H_{B}:=(H, \mu, \hat{\r})$.
\end{example}

In what follows, we show that the category $\mathcal{YD}^{H}_{H}(m,n,p,q)$ can be equipped with a strict braided monoidal structure.
\begin{lemma}\label{L3.3}
Let $V, W\in \mathcal{YD}^{H}_{H}(m,n,p,q)$. Then $V\o W\in \mathcal{YD}^{H}_{H}(m,n,p,q)$ with the structures defined as follows. For any $h\in H$, $v\in V$ and $w\in W$, 
\begin{align*}
(v\o w)\tl h&=v\tl \om^{-1}(h_{1})\o w\tl \psi^{-1}(h_{2}),
\\\r^{r}(v\o w)=&v_{(0)}\o n_{(0)}\o \a^{-1}(v_{(1)})\b^{-1} (w_{(1)}),
\end{align*}
\end{lemma}
The proof is not hard and the interested readers can refer to the calculation of Theorem 4.4 in \cite{ZWC24}.

\begin{theorem}\label{T3.5}
Let $H$ be a BiHom-Hopf algebra with a bijective antipode $S$. Then the category $\mathcal{YD}^{H}_{H}(m,n,p,q)$ is a strict braided monoidal category, with tensor product $\o$ defined as in Lemma \ref{L3.3}. Its braiding structure is defined by
\begin{align*}
&c_{V,W}:V\o W\ra W\o V,
\\&c_{V,W}(v\o w)=\om_{W}^{-1}(w_{(0)})\o \a_{V}^{-1}(v)\tl \a^{-m+1}\b^{-n}\om^{-p+1}\psi^{-q}(w_{(1)})
\end{align*}
and the inverse
\begin{align*}
&c_{V,W}^{-1}:W\o V\ra V\o W,
\\&c_{V,W}^{-1}(w\o v)=\a^{-1}_{V}(v)\tl S^{-1}\a^{-m}\b^{-n+1}\om^{-p}\psi^{-q+1}(w_{(1)})\o  \om^{-1}_{W}(w_{(0)}),
\end{align*}
for any $v\in V$ and $w\in W$.
\end{theorem}
We can refer to the calculation of Theorem 4.4 in \cite{ZWC24}.
\section{A category equivalence}
\def\theequation{4.\arabic{equation}}
\setcounter{equation} {0}
In this section, let $H$ be a BiHom-Hopf algebra, we establish an equivalence between the monoidal category $(~\!^{H}_{H}\mathfrak{M}^{H}_{H},\o_{H})$ and $(~\!^{H}_{H}\mathfrak{M}^{H}_{H},\B_{H})$, and the monoidal category $\mathcal{YD}^{H}_{H}(m,n,p,q)$ over $H$. This equivalence generalizes the main result in \cite{S94}.

\begin{lemma}\label{L2}
Let $V$ be a linear space and $\a_{V},\b_{V},\om_{V},\psi_{V}\in Aut(V)$. Endow the object $H\o V\in \!^{H}_{H}\mathfrak{M}$ with the structures given as in Example $\ref{E1}$ $(1)$. Then there is a bijection between
\begin{enumerate}
\item[$(1)$] a right $H$-comodule structures on $H\o V$ making $H\o V$ an object of $\!^{H}_{H}\mathfrak{M}^{H}$;
\item[$(2)$] a right $H$-comodule structures on $ V$ making $\e\o V:V\ra H\o V$ a morphism of right $H$-comodules.
\end{enumerate}
\end{lemma}
\begin{proof} $(2)\Rightarrow (1)$
If $V$ is a right $H$-comodule, for any $h \in H$ and $v\in V$, we define
\begin{align}\label{e4.6}
\r^{r}:H\o V\ra H\o V\o H, \r^{r}(g\o v)=g_{1}\o v_{(0)}\o \a^{-1}(g_{2}v_{(1)}),
\end{align}
for any $g\in H$ and $v\in V$.
By the definition of two-cosided BiHom-Hopf module, we only need to prove that $H\o V$ is a right $H$-comodule and the object $H\o V\in \!_{H}\mathfrak{M}^{H}$. We first prove that $H\o V$ is a right $H$-comodule. For any $h \in H$ and $v\in V$, we have
\begin{align*}
(\r^{r}\o \psi)\r^{r}(h\o v)&=(\r^{r}\o \psi)(h_{1}\o v_{(0)}\o \a^{-1}(h_{2}v_{(1)}))
\\&=h_{11}\o v_{(0)(0)}\o \a^{-1}(h_{12}v_{(0)(1)})\o \a^{-1}\psi (h_{2}v_{(1)})
\\&=\om(h_{1})\o \om_{V}(v_{(0)})\o \a^{-1}(h_{21}v_{(1)1})\o \a^{-1}(h_{22}v_{(1)2})
\\&=(\om\o \om_{V}\o \D)\r^{r}(h\o v),
\end{align*}
It is easy to check that the rest equation hold.

Finally, we verify that the compatibility condition $(\ref{e2.2})$. For any $h,g \in H$ and $v\in V$, we have
\begin{align*}
\r^{r}(h\c (g\o v))&=\r^{r}(hg\o \b_{V}(v))
\\&=h_{1}g_{1}\o \b_{V}(v_{(0)})\o \a^{-1}[(h_{2}g_{2})\b(v_{(1)})]
\\&=h_{1}\c(g_{1}\o v_{(0)})\o h_{2}\a^{-1}(g_{2}v_{(1)})
\\&=h_{1}\c(g\o v)_{(0)}\o h_{2}(g\o v)_{(1)}.
\end{align*}

$(1)\Rightarrow (2)$ If $H\o V\in \!^{H}_{H}\mathfrak{M}^{H}$ with the right $H$-comodule structure $\r^{r}:H\o V\ra H\o V\o H$. Then there is a unique right $H$-comodule structure on $V$ given by 
$$\r=(\v\o V\o H)\circ \r^{r}\circ (\e\o V):V\ra V\o H.$$

First, applying $(\e\o V\o H)$ to both sides of the equation above, we obtain
\begin{align*}
(\e\o V\o H)\circ \r=(\e\o V\o H)\circ(\v\o V\o H)\circ\r^{r}\circ(\e\o V)=\r^{r}\circ(\e\o V).
\end{align*}
Thus $\e\o V$ is a morphism of right $H$-comodules.

Next, we prove that $V$ is a right $H$-comodule. We have
\begin{align*}
(\r\o \psi)\circ\r=&(\v\o V\o H\o H)\circ(\r^{r}\o \psi)\circ(\e\o V\o H)\circ(\v\o V\o H)\circ\r^{r}\circ(\e\o V)
\\&=(\v \o V\o H\o H)\circ(\r^{r}\o \psi)\circ\r^{r}\circ(\e\o V)
\\&=(\v\o V\o H\o H)\circ(\om\o\om_{V}\o \D)\circ\r^{r}\circ(\e\o V)
\\&=(\om_{V}\o \D)\circ(\v\o V\o H)\circ\r^{r}\circ(\e\o V)
\\&=(\om_{V}\o \D)\circ\r
\end{align*}
and the other equation is easy to prove.

This completes the proof.
\end{proof}
The proof of Lemma \ref{L2} was given in the framework of general monoidal categories. Applying the lemma to the opposite category gives:
\begin{lemma}\label{L1}
Let $V$ be a linear space and $\a_{V},\b_{V},\om_{V},\psi_{V}\in Aut(V)$. Endow $H\o V\in \!^{H}_{H}\mathfrak{M}$ with the structures given as in Example $\ref{E1}$ $(1)$. Then there is a bijection between
\begin{enumerate}
\item[$(1)$] right $H$-module structures on $H\o V$ making $H\o V$ an object of $\!^{H}_{H}\mathfrak{M}_{H}$;
\item[$(2)$] right $H$-module structures on $ V$ making $\v\o V:H\o V\ra V$ a morphism of right $H$-modules.
\end{enumerate}
\end{lemma}
If $V$ is a right $H$-module. The induced right $H$-module structure on $H\o V$ are defined as follows, for any $h,g\in H$ and $v\in V$:
\begin{align}
(g\o v)\c h&=g\om^{-1}(h_{1})\o v\tl \a^{-m+2}\b^{-n}\om^{-p+1}\psi^{-q}(h_{2}), \label{e4.5}
\end{align}
where $m,n,p,q\in\mathbb{Z}$. The proof is analogous to that of Lemma \ref{L2} and is omitted for brevity.
\begin{theorem}
Let $V$ be a linear space and $\a_{V},\b_{V},\om_{V},\psi_{V}\in Aut(V)$. Endow the object $H\o V\in \!^{H}_{H}\mathfrak{M}$ with the structures given as in Example $\ref{E1}$ $(1)$. Then there is a bijection between
\begin{enumerate}
\item[$(1)$] a right $H$-module structure and a right $H$-comodule structure on $H\o V$
making $H\o V$ an object of $\!^{H}_{H}\mathfrak{M}^{H}_{H}$;
\item[$(2)$] a structure of right-right BiHom-$(m,n,p,q)$-Yetter-Drinfel'd module on $V$.
\end{enumerate}
\end{theorem}
\begin{proof}
Based on Lemma \ref{L2} and \ref{L1}, we only need to prove that the condition on the right $H$-module structure and right $H$-comodule structure on $V$ that they define a right-right BiHom-$(m,n,p,q)$-Yetter-Drinfel'd module is equivalent to the condition making $H\o V$ an object of $\mathfrak{M}^{H}_{H}$.

Let $(V,\tl)$ be a right $H$-module and $(V, \r)$ a right $H$-comodule. The induced right $H$-module structure and right $H$-comodule structure on $H\o V$ are defined in Eqs. (\ref{e4.5}), (\ref{e4.6}), respectively.

Let $A=\a^{m}\b^{n}\om^{p}\psi^{q}$ with $m,n,p,q\in \mathrm{Z}$. For any $h,g\in H$ and $v\in V$, we have
\begin{align*}
\r^{r}&((g\o v)\c h)=\r^{r}(g\om^{-1}(h_{1})\o v\tl A^{-1}\a^{2}\om(h_{2}))
\\&=g_{1}\om^{-1}(h_{11})\o (v\tl A^{-1}\a^{2}\om(h_{2}))_{(0)}\o \a^{-1}[[g_{2}\om^{-1}(h_{12})]((v\tl A^{-1}\a^{2}\om(h_{2}))_{(1)})]
\\&=g_{1}h_{1}\o (v\tl A^{-1}\a^{2}\om\psi^{-1}(h_{22}))_{(0)}\o \a^{-1}[[g_{2}\om^{-1}(h_{21})]((v\tl A^{-1}\a^{2}\om\psi^{-1}(h_{22}))_{(1)})]
\\&=g_{1}h_{1}\o (v\tl A^{-1}\a^{2}\om\psi^{-1}(h_{22}))_{(0)}\o g_{2}\a^{-1}\b^{-1}[\b\om^{-1}(h_{21})(v\tl A^{-1}\a^{2}\om\psi^{-1}(h_{22}))_{(1)}]
\end{align*}
and
\begin{align*}
(g\o v)_{(0)}&\c h_{1}\o (g\o v)_{(1)}h_{2}
\\&=(g_{1}\o v_{(0)})\c h_{1}\o \a^{-1} (g_{2}v_{(1)})h_{2}
\\&=g_{1}\om^{-1}(h_{11})\o v_{(0)}\tl A^{-1}\a^{2}\om(h_{12})\o g_{2}[\a^{-1} (v_{(1)})\b^{-1}(h_{2})]
\\&=g_{1}h_{1}\o v_{(0)}\tl A^{-1}\a^{2}\om(h_{21})\o g_{2}\a^{-1}\b^{-1}[\b (v_{(1)})\a\psi^{-1}(h_{22})].
\end{align*}

If $V$ is a right-right BiHom-$(m,n,p,q)$-Yetter-Drinfel'd module, for Eq. $(\ref{e3.1}):$
$$[v\tl h_2]_{(0)}\o A\alpha^{-2}\b\omega^{-2}\psi(h_{1})[v\tl h_2]_{(1)}=v_{(0)}\tl \psi(h_{1})\o\b(v_{(1)})A\alpha^{-1}\omega^{-1}(h_{2}).$$
Replace $h$ by $A^{-1}\a^{2}\om\psi^{-1}(h_{2})$. Then one easily sees that these two terms are equal. Thus $H\o V$ is a right-right BiHom-Hopf module over $H$.

Conversely, assuming that $H\o V$ is an object of $\mathfrak{M}^{H}_{H}$, that is, for any $h\in H$ and $v\in V$, by $\r^{r}((1\o v)\c h)=(1\o v)_{(0)}\c h_{1}\o (1\o v)_{(1)}h_{2}$, we get
\begin{align*}
\b(h_{1})\o &(v\tl A^{-1}\a^{2}\om\psi^{-1}(h_{22}))_{(0)}\o \a^{-1}[\b\om^{-1}(h_{21})(v\tl A^{-1}\a^{2}\om\psi^{-1}(h_{22}))_{(1)}]
\\&=\b(h_{1})\o v_{(0)}\tl A^{-1}\a^{2}\om(h_{21})\o \a^{-1}[\b (v_{(1)})\a\psi^{-1}(h_{22})],
\end{align*}
applying $\v\o V\o \a $ to both sides of the equation above, we get
\begin{align*}
(v\tl A^{-1}\a^{2}\om(h_{2}))_{(0)}\o \b\om^{-1}\psi(h_{1})(v\tl A^{-1}\a^{2}\om(h_{2}))_{(1)}
=v_{(0)}\tl A^{-1}\a^{2}\om\psi(h_{1})\o \b (v_{(1)})\a(h_{2}),
\end{align*}
replacing $h$ by $A\a^{-2}\om^{-1}(h)$, we obtain Eq. (\ref{e3.1}). Thus $V$ is a right-right BiHom-$(m,n,p,q)$-Yetter-Drinfel'd module over $H$.

This completes the proof.
\end{proof}
\begin{theorem}\label{T1}
Let $H$ be a BiHom-Hopf algebra. Then the equivalence
\begin{align*}
\!^{H}_{H}\mathfrak{M}&\cong \mathfrak{C}
\\H\o V&\leftarrow V
\\ M&\ra \!^{coH}M
\end{align*}
induces equivalences of monoidal categories between
\begin{enumerate}
\item[$(1)$] the category $\!^{H}_{H}\mathfrak{M}_{H}$ of two-sided BiHom-Hopf modules with BiHom-tensor product $\o_{H}$ and the category of right $H$-modules,
\item[$(2)$] the category $\!^{H}_{H}\mathfrak{M}^{H}$ of two-cosided BiHom-Hopf modules with BiHom-cotensor product $\B_{H}$ and the category of right $H$-comodules,
\item[$(3)$] the category $\!^{H}_{H}\mathfrak{M}^{H}_{H}$ of BiHom-four-angle Hopf modules with either $\o_{H}$ or $\B_{H}$ as its monoidal product, and the category of right-right BiHom-$(m,n,p,q)$-Yetter-Drinfel'd modules over $H$,
\end{enumerate}
where the right $H$-$(co)$module structures on $H\otimes V$ (for $V$ a right $H$-$(co)$module) are given by Eqs. $(\ref{e4.5})$ and $(\ref{e4.6})$.
\end{theorem}
\begin{proof}We define the subspace of $M$ by
\begin{align*}
\!^{coH}M=\{m\in M~|~\r^{l}(m)=1\o \psi_{M}(m)\}.
\end{align*}
The right $H$-comodule structure on $\!^{coH}M$ for $M\in \!^{H}_{H}\mathfrak{M}^{H}$ is that of ${}^{coH}M$ as a right $H$-subcomodule of $M$. The right $H$-module structure on $\!^{coH}M$ for $M\in \!^{H}_{H}\mathfrak{M}_{H}$ is defined by, for any $h\in H$ and $m'\in \!^{coH}M$,
$$m'\tl h=S\a^{m-3}\b^{n+1}\omega^{p-3}\psi^{q}(h_{1})\c(\b_{M}^{-1}(m')\c \a^{m-3}\b^{n}\omega^{p-3}\psi^{q}(h_{2})), $$
Let $A=\a^{m}\b^{n}\omega^{p}\psi^{q}$ with $m,n,p,q\in \mathrm{Z}$. We first check that the action is well defined,
\begin{align*}
\r^{l}(m'\tl h)&=\r^{l}(SA\a^{-3}\b\omega^{-3}(h_{1})\c (\b_{M}^{-1}(m')\c A\a^{-3}\omega^{-3}(h_{2})))
\\&=SA\a^{-3}\b\omega^{-3}(h_{1})_{1} (\b_{M}^{-1}(m')_{[-1]}A\a^{-3}\omega^{-3}(h_{21}))
\\&\quad\o SA\a^{-3}\b\omega^{-3}(h_{1})_{2}\c (\b_{M}^{-1}(m')_{[0]}\c A\a^{-3}\omega^{-3}(h_{22}))
\\&=SA\a^{-3}\b\omega^{-2}\psi^{-1}(h_{12}) A\a^{-3}\b\omega^{-3} (h_{21})
\\&\quad\o SA\a^{-3}\b\omega^{-4}\psi(h_{11})\c (\b_{M}^{-1}\psi_{M}(m')\c A\a^{-3}\omega^{-3}(h_{22}))
\\&=SA\a^{-3}\b\omega^{-3}\psi^{-1}(h_{211})A\a^{-3}\b\omega^{-3} \psi^{-1}(h_{212})
\\&\quad\o SA\a^{-3}\b\omega^{-3}\psi(h_{1})\c (\b_{M}^{-1}\psi_{M}(m')\c A\a^{-3}\omega^{-3}(h_{22}))
\\&=1\o SA\a^{-3}\b\omega^{-3}\psi(h_{1})\c (\b_{M}^{-1}\psi_{M}(m')\c A\a^{-3}\omega^{-3}\psi(h_{2}))
\\&=1\o \psi_{M}(m')\tl \psi (h)
\\&=1\o \psi_{M}(m'\tl h).
\end{align*}
It is easy to show that ${}^{coH}M\in \mathfrak{M}_{H}$.

Next, we only need to check the assertion that we have monoidal equivalences. To do this, it is enough to prove that one of the quasi-inverse equivalences is a monoidal functor in each case.

Let $T=\a^{a}\b^{b}\om^{c}\psi^{d}, T_{V}=\a_{V}^{a}\b_{V}^{b}\om_{V}^{c}\psi_{V}^{d}, T_{W}=\a_{W}^{a}\b_{W}^{b}\om_{W}^{c}\psi_{W}^{d}$ with $a,b,c,d\in\mathbb{Z}$. For $(1)$ we show that the map
\begin{align*}
\vp:(H\o V)\o_{H}(H\o W) &\ra H\o V\o W
\\ g\o v\o_{H} h\o w &\m(T(g)\o T_{V}(v))\c T(h)\o T_{W}\b_{W}(w)
\end{align*}
is a morphism of two-sided BiHom-Hopf modules.
For left linearity and colinearity, by computing we have
\begin{align*}
\vp[k\c &(g\o v\o_{H} h\o w)]=\vp[T^{-1}\a^{-1}(k)\c (g\o v)\o_{H} \b(h)\o \b_{W}(w)]
\\&=\vp[T^{-1}\a^{-1}(k)g\o \b_{V}(v)\o_{H} \b(h)\o \b_{W}(w)]
\\&=T(T^{-1}\a^{-1}(k)g)T\b\om^{-1}(h_{1})\o T_{V}\b_{V}(v)\tl TA^{-1}\a^{2}\b\om(h_{2})\o T_{W}\b_{W}^{2}(w)
\\&=k[T(g)T\om^{-1}(h_{1})]\o T_{V}\b_{V}(v)\tl TA^{-1}\a^{2}\b\om(h_{2})\o T_{W}\b_{W}^{2}(w)
\\&=k\c [T(g)T\om^{-1}(h_{1})\o T_{V}(v)\tl TA^{-1}\a^{2}\om(h_{2})\o T_{W}\b_{W}(w)]
\\&=k\c [\vp(g\o v\o_{H} h\o w)]
\end{align*}
and
\begin{align*}
[&\vp(g\o v\o_{H} h\o w)]_{[-1]}\o [\vp(g\o v\o_{H} h\o w)]_{[0]}
\\&=(T(g_{1})T\om^{-1}(h_{11}))\o T(g_{2})T\om^{-1}(h_{12})\o T_{V}\psi_{V}(v)\tl TA^{-1}\a^{2}\om\psi(h_{2})\o T_{W}\b_{W}\psi_{W}(w)
\\&=T(g_{1})T(h_{1})\o T(g_{2})T\om^{-1}(h_{21})\o T_{V}\psi_{V}(v)\tl TA^{-1}\a^{2}\om(h_{22})\o T_{W}\b_{W}\psi_{W}(w)
\\&=T(g_{1}h_{1})\o T(g_{2})T\om^{-1}(h_{21})\o T_{V}\psi_{V}(v)\tl TA^{-1}\a^{2}\om(h_{22})\o T_{W}\b_{W}\psi_{W}(w)
\\&=T(g_{1}h_{1}) \o \vp[g_{2}\o \psi_{V}(v)\o_{H} h_{2}\o \psi_{W}(w)  ],
\\&=(g\o v\o h\o w)_{[-1]} \o \vp[(g\o v\o_{H} h\o w)_{[0]}],
\end{align*}
for any $g,h,k\in H$, $v\in V$ and $w\in W$.
For right linearity, we have
\begin{align*}
\vp[(g\o v&\o_{H} h\o w)\c k]
\\&=\vp[\a(g)\o \a_{V}(v))\o (h\o w)\c T^{-1}\b^{-1}(k)]
\\&=\vp[\a(g)\o \a_{V}(v)\o hT^{-1}\b^{-1}\om^{-1}(k_{1})\o w\tl A^{-1}T^{-1}\a^{2}\b^{-1}\om (k_{2})]
\\&=T\a(g)T\om^{-1}(h_{1}T^{-1}\b^{-1}\om^{-1}(k_{11}))
\\&\quad\o T_{V}\a_{V}(v)\tl TA^{-1}\a^{2}\om(h_{2}\om^{-1}T^{-1}\b^{-1}(k_{12}))\o T_{W}\b_{W}(w\tl A^{-1}T^{-1}\a^{2}\b^{-1}\om (k_{2}))
\\&=[T(g)T\om^{-1}(h_{1})]\om^{-2}(k_{11})
\\&\quad\o [T_{V}(v)\tl TA^{-1}\a^{2}\om(h_{2})]\tl A^{-1}\a^{2}(k_{12})\o T_{W}\b_{W}(w)\tl A^{-1}\a^{2}\om (k_{2}))
\\&=[T(g)T\om^{-1}(h_{1})]\om^{-1}(k_{1})
\\&\quad\o [T_{V}(v)\tl TA^{-1}\a^{2}\om(h_{2})]\tl A^{-1}\a^{2}(k_{21})\o T_{W}\b_{W}(w)\tl A^{-1}\a^{2}\om\psi^{-1}(k_{22})
\\&=[T(g)T\om^{-1}(h_{1})]\om^{-1}(k_{1})\o [T_{V}(v)\tl TA^{-1}\a^{2}\om(h_{2})\o T_{W}\b_{W}(w)]\tl A^{-1}\a^{2}\om(k_{2})
\\&=[T(g)T\om^{-1}(h_{1})\o T_{V}(v)\tl TA^{-1}\a^{2}\om(h_{2})\o T_{W}\b_{W}(w)]\c k
\\&=[\vp(g\o v\o_{H} h\o w)]\c k
\end{align*}

Furthermore, we show that $\vp$ is an isomorphism whose inverse is given by
\begin{align*}
\vp^{-1}:H\o V\o W&\ra (H\o V)\o_{H}(H\o W) 
\\ g\o v\o w &\m T^{-1}\a^{-1}(g)\o T_{V}^{-1}\a_{V}^{-1}(v) \o 1\o T_{W}^{-1}\b_{W}^{-1}(w),
\end{align*}
for any $g,h\in H$, $v\in V$ and $w\in W$, we get
\begin{align*}
\vp^{-1}\circ &\vp(g\o v\o_{H} h\o w)
\\&=T^{-1}\a^{-1}[T(g)T\om^{-1}(h_{1})]\o T_{V}^{-1}\a_{V}^{-1}[T(v)\tl TA^{-1}\a^{2}\om(h_{2})]\o_{H} 1\o w
\\&=\a^{-1}(g)\a^{-1}\om^{-1}(h_{1})\o \a_{V}^{-1}(v)\tl A^{-1}\a\om(h_{2})\o_{H} 1\o w
\\&=[\a^{-1}(g)\o \a_{V}^{-1}(v)]\c \a^{-1}(h)\o_{H} 1\o w
\\&=g\o v\o_{H} \a^{-1}(h)\c (1\o \b_{W}^{-1}(w))
\\&=g\o v\o_{H} \a^{-1}(h)1\o w
\\&=g\o v\o_{H} h\o w
\end{align*}
and
\begin{align*}
\vp\circ \vp^{-1}(g\o v\o w)&=\a^{-1}(g)1_{H}\o \a_{V}^{-1}(v)\tl 1_{H}\o w=g\o v\o w.
\end{align*}

Part $(2)$ is formally dual to $(1)$. We only deal with the half of $(3)$ involving $\o_{H}$ since the other half is dual to this. It remains to check that $\vp$ is right colinearity. For any $g,h\in H$, $v\in V$ and $w\in W$, we get
\begin{align*}
&[\vp(g\o v\o_{H} h\o w)]_{(0)}\o [\vp(g\o v\o_{H} h\o w)]_{(1)}
\\&=T(g_{1})T\om^{-1}(h_{11})\o (T_{V}(v)\tl TA^{-1}\a^{2}\om(h_{2}))_{(0)}\o T_{W}\b_{W}(w_{(0)})
\\&\quad\o \a^{-1}(T(g_{2})T\om^{-1}(h_{12}))\a^{-1}[\a^{-1}(T_{V}(v)\tl TA^{-1}\a^{2}\om(h_{2}))_{(1)}T(w_{(1)})]
\\&=T(g_{1})T\om^{-1}(h_{11})\o (T_{V}(v)\tl TA^{-1}\a^{2}\om(h_{2}))_{(0)}\o T_{W}\b_{W}(w_{(0)})
\\&\quad\o T(g_{2})[[T\a^{-2}\om^{-1}(h_{12})\a^{-2}\b^{-1}(T_{V}(v)\tl TA^{-1}\a^{2}\om(h_{2}))_{(1)}]T\a^{-1}(w_{(1)})]
\\&=T(g_{1})T(h_{1})\o [T_{V}(v)\tl TA^{-1}\a^{2}\om\psi^{-1}(h_{22})]_{(0)}\o T_{W}\b_{W}(w_{(0)})
\\&\quad\o T(g_{2})[\a^{-2}\b^{-1}[T\b\om^{-1}(h_{21})[T_{V}(v)\tl TA^{-1}\a^{2}\om\psi^{-1}(h_{22})]_{(1)}]
T\a^{-1}(w_{(1)})]
\end{align*}
and
\begin{align*}
&\vp[(g\o v\o_{H} h\o w)_{(0)}]\o (g\o v\o_{H} h\o w)_{(1)}
\\&=\vp[g_{1}\o v_{(0)}\o_{H} h_{1}\o w_{(0)}]\o T(\a^{-1}(g_{2}v_{(1)}))T(\a^{-1}(h_{2}w_{(1)}))
\\&=T(g_{1})T\om^{-1}(h_{11})\o T_{V}( v_{(0)})\tl TA^{-1}\a^{2}\om(h_{12})\o T_{W}\b_{W}(w_{(0)})
\\&\quad\o T(g_{2})[[T\a^{-2}(v_{(1)})T\a^{-1}\b^{-1}(h_{2})]T\a^{-1}(w_{(1)})]
\\&=T(g_{1})T(h_{1})\o T_{V}( v_{(0)})\tl TA^{-1}\a^{2}\om(h_{21})\o T_{W}\b_{W}(w_{(0)})
\\&\quad\o T(g_{2})[[T\a^{-2}(v_{(1)})T\a^{-1}\b^{-1}\psi^{-1}(h_{22})]T\a^{-1}(w_{(1)})]
\\&=T(g_{1})T(h_{1})\o T_{V}( v_{(0)})\tl TA^{-1}\a^{2}\om(h_{21})\o T_{W}\b_{W}(w_{(0)})
\\&\quad\o T(g_{2})[\a^{-2}\b^{-1}[T\b(v_{(1)})T\a\psi^{-1}(h_{22})]T\a^{-1}(w_{(1)})].
\end{align*}

By Eq. $(\ref{e3.1}):$
$$[v\tl h_2]_{(0)}\o A\alpha^{-2}\b\omega^{-2}\psi(h_{1})[v\tl h_2]_{(1)}=v_{(0)}\tl \psi(h_{1})\o\b(v_{(1)})A\alpha^{-1}\omega^{-1}(h_{2}).$$
Replace $v$ by $T_{V}(v)$ and $h$ by $TA^{-1}\a^{2}\om\psi^{-1}(h_{2})$. Then one easily sees that these two terms are equal. 

Finally, the coherence condition on monoidal functors follows from the fact that both ways around the rectangle
$$\xymatrix{
  (H\o U)\o_{H}(H\o V)\o_{H}(H\o W) \ar[d]_{\vp\o_{H}\mathrm{id}} \ar[r]^-{\mathrm{id}\o_{H}\vp}      & (H\o U)\o_{H}(H\o V\o W)\ar[d]^{\vp}  \\
  (H\o U\o V)\o_{H}(H\o W)  \ar[r]_-{\vp}               & H\o U\o V\o W            }
  $$
are given by
$T(h)[T^{2}\om^{-1}(g_{1})T^{2}\om^{-2}(f_{11})]\o T_{U}(u)\tl [T^{2}A^{-1}\a^{2}\om(g_{2})T^{2}A^{-1}\a^{2}(f_{12})]
\o T_{V}^{2}\b_{V}(v)\tl  T^{2}A^{-1}\a^{2}\b\om(f_{2})\o T_{W}^{2}\b_{W}^{2}(w)$.

This completes the proof.
\end{proof}
\begin{remark}
(1) If $H$ is a Hom-Hopf algebra, i.e., $\alpha = \beta= \omega = \psi $, and $T=\b^{-1}, A=\alpha\beta\omega\psi$, we obtain an equivalence between the monoidal category $(~\!^{H}_{H}\mathfrak{M}^{H}_{H},\otimes_{H})$ or $(~\!^{H}_{H}\mathfrak{M}^{H}_{H},\Box_{H})$ of four-angle Hopf modules and the monoidal category $\mathcal{YD}^{H}_{H}$ of Yetter-Drinfel'd modules over $H$, This result was introduced in \cite{LYDW26}.

(2) If $H$ is a monoidal Hom-Hopf algebra, i.e., $\omega = \psi = \alpha^{-1} = \beta^{-1}$, and $A=\alpha\beta\omega\psi$, we obtain an equivalence between the monoidal category $(~\!^{H}_{H}\mathfrak{M}^{H}_{H},\otimes_{H})$ or $(~\!^{H}_{H}\mathfrak{M}^{H}_{H},\Box_{H})$ of four-angle Hopf modules and the monoidal category $\mathcal{YD}^{H}_{H}$ of Yetter-Drinfel'd modules over $H$.

(3) If $H$ is a Hopf algebra, i.e., $\alpha = \beta= \omega = \psi=id $, we obtain an equivalence between the monoidal category $(~\!^{H}_{H}\mathfrak{M}^{H}_{H},\otimes_{H})$ or $(~\!^{H}_{H}\mathfrak{M}^{H}_{H},\Box_{H})$ of four-angle Hopf modules and the monoidal category $\mathcal{YD}^{H}_{H}$ of Yetter-Drinfel'd modules over $H$, This result was introduced in \cite{S94}.
\end{remark}
\begin{example}
Let $H_{4}=sp\{1,g, x,gx\}$ be a vector space over $\Bbbk$ with char $\Bbbk\neq 2$ satisfying the following relation:
\begin{align*}
g^{2}=1,x^{2}=0,xg=-gx.
\end{align*}

Let $\forall 0\neq a,b,c,d\in \Bbbk$ with $ac=bd$, define the BiHom-Hopf algebra structure on $H_{4}$ as follows:
\begin{enumerate}
\item[$\bullet$] The automorphism $\a,\b,\om,\psi: H_{4}\ra H_{4}$ is given by
\begin{align*}
&\a(1)=1, \quad \a(g)=g, \quad \a(x)=ax, \quad \a(gx)=agx,
\\&\b(1)=1, \quad \b(g)=g, \quad \b(x)=bx, \quad \b(gx)=bgx,
\\&\om(1)=1, \quad \om(g)=g, \quad \om(x)=cx, \quad \om(gx)=cgx,
\\&\psi(1)=1, \quad \psi(g)=g, \quad \psi(x)=dx, \quad \psi(gx)=dgx.
\end{align*}

\item[$\bullet$] The multiplication $\circ$ is given by:
\begin{align*}
\begin{tabular}{c|c c c c}
	       H&1&g&x&gx\\
	\hline 1&1&g&bx&bgx\\
	       g&g&1&bgx&bx\\
	       x&ax&$-$agx&0&0\\
           gx&agx&$-$ax&0&0\\
\end{tabular}.
\end{align*}
\item[$\bullet$] The comultiplication $\D$, counit $\v$ and antipode $S$ are given by:
\begin{align*}
&\D(1)=1\o 1, \quad\D(g)=g\o g, \quad
\\& \D(x)=cx\o g+1\o dx,\quad \D(gx)=cgx\o 1+g\o dgx,
\\&\v(1)=1_{\Bbbk},\quad\v(g)=1_{\Bbbk},\quad \v(x)=0, \quad\v(gx)=0,
\\&S(1)=1,\quad S(g)=g,\quad S(x)=gx, \quad S(gx)=-x.
\end{align*}
\end{enumerate}

Let $V=sp\{1_{V}, z\}$ over $\Bbbk$ with char $\Bbbk\neq 2$ and define the automorphism $\a_{V}, \b_{V},\om_{V},\psi_{V}: V\ra V$ by
\begin{align*}
&\a_{V}(1_{V})=\b_{V}(1_{V})=\om_{V}(1_{V})=\psi_{V}(1_{V})=1_{V},
\\&\a_{V}(z)=az,\quad \b_{V}(z)=bz,\quad\om_{V}(z)=cz,\quad\psi_{V}(z)=dz,
\end{align*}
Define the action $\tl:V\o H_{4}\ra V$ by
\begin{align*}
&1_{V}\tl 1=1_{V},\quad 1_{V}\tl g=1_{V},\quad 1_{V}\tl x=0,\quad 1_{V}\tl gx=0,
\\&z\tl 1=az,\quad z\tl g=-az,\quad z\tl x=0,\quad z\tl gx=0.
\end{align*}
Define the coaction $\r: V\ra V\o H_{4}$ by
\begin{align*}
&\r(1_{V})=1_{V}\o 1,\quad \r(z)=cz\o g+1_{V}\o dx.
\end{align*}
Then $(V, \z_{V})$ is a right-right BiHom-$(m,n,p,q)$-Yetter-Drinfel'd module,
where $0\neq a,b,c,d\in \Bbbk$. Then one can check that $(H_{4}\o V, \b\o \z_{V})$ is a BiHom-four-angle Hopf module with the following structures:
\begin{enumerate}
\item[$(i)$] the left module structures:
\begin{align*}
\begin{tabular}{c|c c c c}
	\hline $\c$&$1\o 1_{V}$&$g\o 1_{V}$&$x\o 1_{V}$&$gx\o 1_{V}$\\
	\hline 1&$1\o 1_{V}$&$g\o 1_{V}$&$bx\o 1_{V}$&$bgx\o 1_{V}$\\
	       g&$g\o 1_{V}$&$1\o 1_{V}$&$bgx\o 1_{V}$&$bx\o 1_{V}$\\
	       x&$ax\o 1_{V}$&$-agx\o 1_{V}$&0&0\\
           gx&$agx\o 1_{V}$&$-ax\o 1_{V}$&0&0\\
    \hline $\c$&$1\o z$&$g\o z$&$x\o z$&$gx\o z$\\
	\hline 1&$1\o bz$&$g\o bz$&$bx\o bz$&$bgx\o bz$\\
	       g&$g\o bz$&$1\o bz$&$bgx\o bz$&$bx\o bz$\\
	       x&$ax\o bz$&$-agx\o bz$&0&0\\
           gx&$agx\o bz$&$-ax\o bz$&0&0\\
\end{tabular}.
\end{align*}
\item[$(ii)$] the right module structures:
\begin{align*}
\begin{tabular}{c c c c|c}
	\hline $1\o 1_{V}$&$g\o 1_{V}$&$x\o 1_{V}$&$gx\o 1_{V}$&$\c$\\
	\hline $1\o 1_{V}$&$g\o 1_{V}$&$ax\o 1_{V}$&$agx\o 1_{V}$&1\\
	       $g\o 1_{V}$&$1\o 1_{V}$&$-agx\o 1_{V}$&$-ax\o 1_{V}$&g\\
	       $bx\o 1_{V}$&$bgx\o 1_{V}$&0&0&x\\
           $bgx\o 1_{V}$&$bx\o 1_{V}$&0&0&gx\\
    \hline $1\o z$&$g\o z$&$x\o z$&$gx\o z$&$\c$\\
	\hline $1\o az$&$g\o az$&$ax\o az$&$agx\o az$&1\\
	       $g\o -az$&$1\o -az$&$-agx\o -az$&$-ax\o -az$&g\\
	        $bx\o -az$&$bgx\o -az$&0&0&x\\
           $bgx\o az$&$bx\o az$&0&0&gx\\
\end{tabular}.
\end{align*}
\item[$(iii)$] the left comodule structures:
\begin{align*}
&\r^{l}(1\o 1_{V})=1\o 1\o 1_{V},\quad \r^{l}(1\o z)=1\o 1\o dz,
\\&\r^{l}(g\o 1_{V})=g\o g\o 1_{V},\quad \r^{l}(g\o z)=g\o g\o dz,
\\&\r^{l}(x\o 1_{V})=(cx\o g+1\o dx)\o 1_{V},\quad \r^{l}(x\o z)=(cx\o g+1\o dx)\o dz,
\\&\r^{l}(gx\o 1_{V})=(cgx\o 1+g\o dgx)\o 1_{V},,\quad \r^{l}(gx\o z)=(cgx\o 1+g\o dgx)\o dz.
\end{align*}
\item[$(iv)$] the right comodule structures:
\begin{align*}
&\r^{r}(1\o 1_{V})=1\o 1_{V}\o 1,\quad \r^{r}(1\o z)=1\o cz\o g+1\o 1_{V}\o a^{-1}bdx,
\\&\r^{r}(g\o 1_{V})=g\o 1_{V}\o g,\quad \r^{r}(g\o z)=g\o cz\o 1+g\o 1_{V}\o a^{-1}bdgx,
\\&\r^{r}(x\o 1_{V})=cx\o 1_{V}\o g+1\o 1_{V}\o dx,
\\&\r^{r}(x\o z)=cx\o cz\o 1+cx\o 1_{V}\o a^{-1}bdgx+1\o cz\o -dgx,
\\&\r^{r}(gx\o 1_{V})=cgx\o 1_{V}\o 1+g\o 1_{V}\o dgx,
\\&\r^{r}(gx\o z)=cgx\o cz\o g+cgx\o 1_{V}\o a^{-1}bdx+g\o cz\o -dx.
\end{align*}
\end{enumerate}
We define
$\!^{coH}(H_{4}\o V)=\{h\o v|h_{1}\o h_{2}\o \psi_{V}(v)=1\o \psi(h)\o \psi_{V}(v)\}$. That is
\begin{align*}
&\!^{coH}(H_{4}\o V)=\{1\o 1_{V}, 1\o z\}.
\end{align*}
Then it is not hard to check that $\!^{coH}(H_{4}\o V)$ and $V$ is an isomorphism of right-right BiHom-$(m,n,p,q)$-Yetter-Drinfel'd modules.

\end{example}
\begin{corollary}
Let $H$ be a BiHom-Hopf algebra . Then the identity functor is a monoidal equivalence
\begin{align*}
(~\mathfrak{ID},\xi~): (~\!^{H}_{H}\mathfrak{M}\!^{H}_{H},\B_{H})\ra (~\!^{H}_{H}\mathfrak{M}\!^{H}_{H},\o_{H}),
\end{align*}
where the isomorphisms $\xi:M\o_{H} N\ra M\B_{H} N$ satisfy $\xi (m\o n)=\hat{T}_{M}T_{M}(m)_{(0)}\c \hat{T}_{N}T_{N}(n)_{[-1]}\o \hat{T}_{M}T_{M}(m)_{(1)}\c \hat{T}_{N}T_{N}(n)_{[0]}$, for any $m\in M$ and $n\in N$.
\end{corollary}
\begin{proof}
The identity functor is isomorphic to the composition
$$
\xymatrix@C=0.5cm{
 (~\!^{H}_{H}\mathfrak{M}^{H}_{H},\B_{H}) \ar[rr]^-{\!^{coH}(-)} &&  (\mathcal{YD}^{H}_{H}(m,n,p,q),\o) \ar[rr]^-{H\o (-)} &&  (~\!^{H}_{H}\mathfrak{M}^{H}_{H},\o_{H})  }
  $$
of two monoidal equivalences. We only need to prove that the induced structure of monoidal functor on the identity has the form one have claimed. It is sufficient to consider the case $M=H\o V$ and $N=H\o W$ with $V, W\in \mathcal{YD}^{H}_{H}(m,n,p,q)$. Then $\xi$ is the composition
$$
\xymatrix@C=0.5cm{
(H\o V)\o_{H} (H\o W)\ar[r]^-{\vp} &  (H\o V\o W) \ar[r]^-{\d^{-1}} & (H\o V)\B_{H} (H\o W)  },
  $$
where $\d^{-1}$ is dual to $\vp$ and is defined by $\d^{-1}(g\o v\o w)=\hat{T}(g_{1})\o \hat{T}_{V}(v_{(0)})\B_{H} \hat{T}\a^{-1}(g_{2}v_{(1)})\o \hat{T}_{W}\psi_{W}(w)$. Thus we have
\begin{align*}
\d^{-1}\vp&(h\o v\o_{H}g\o w)
\\&=\d^{-1}(T(g)T\om^{-1}(h_{1})\o T_{V}(v)\tl  TA^{-1}\a^{2}\om(h_{2})\o T_{W}\b_{W}(w))
\\&=\hat{T}(T(g_{1})T\om^{-1}(h_{11}))\o \hat{T}_{V}((T_{V}(v)\tl  TA^{-1}\a^{2}\om(h_{2}))_{(0)})
\\&\quad\B_{H} \hat{T}\a^{-1}((T(g_{2})T\om^{-1}(h_{12}))(T_{V}(v)\tl  TA^{-1}\a^{2}\om(h_{2}))_{(1)})\o \hat{T}_{W}T_{W}\b_{W}\psi_{W}(w)
\\&=\hat{T}T(g_{1})\hat{T}T(h_{1})\o \hat{T}_{V}((T_{V}(v)\tl  TA^{-1}\a^{2}\om\psi^{-1}(h_{22}))_{(0)})
\\&\quad\B_{H} \hat{T}T(g_{2})\hat{T}\a^{-1}\b^{-1}[T\b\om^{-1}(h_{21})(T_{V}(v)\tl  TA^{-1}\a^{2}\om\psi^{-1}(h_{22}))_{(1)}]\o \hat{T}_{W}T_{W}\b_{W}\psi_{W}(w)
\\&=\hat{T}T(g_{1})\hat{T}T(h_{1})\o \hat{T}_{V}(T_{V}(v_{(0)})\tl TA^{-1}\a^{2}\om(h_{21}))
\\&\quad\B_{H} \hat{T}T(g_{2})\hat{T}\a^{-1}\b^{-1}[T\b(v_{(1)})T\a\psi^{-1}(h_{22})]\o \hat{T}_{W}T_{W}\b_{W}\psi_{W}(w)
\\&=\hat{T}T(g_{1})\hat{T}T\om^{-1}(h_{11})\o \hat{T}_{V}T_{V}(v_{(0)})\tl \hat{T}TA^{-1}\a^{2}\om(h_{12})
\\&\quad\B_{H} (\hat{T}T\a^{-1}(g_{2})\hat{T}T\a^{-1}(v_{(1)}))\hat{T}T(h_{2})\o \hat{T}_{W}T_{W}\b_{W}\psi_{W}(w)
\\&=(\hat{T}T(g_{1})\o \hat{T}_{V}T_{V}(v_{(0)}))\c \hat{T}T(h_{1})
\\&\quad\B_{H} (\hat{T}T\a^{-1}(g_{2})\hat{T}T\a^{-1}(v_{(1)}))\c(\hat{T}T(h_{2})\o \hat{T}_{W}T_{W}\psi_{W}(w))
\\&=(\hat{T}T(g)\o \hat{T}_{V}T_{V}(v))_{(0)}\c (\hat{T}T(h)\o \hat{T}_{W}T_{W}(w))_{[-1]}
\\&\quad\B_{H} (\hat{T}T(g)\o \hat{T}_{V}T_{V}(v))_{(1)}\c(\hat{T}T(h)\o \hat{T}_{W}T_{W}(w))_{[0]}.
\end{align*}

The fourth equality holds by Eq. $(\ref{e3.1})$:
\begin{align*}
&[v\tl h_2]_{(0)}\o A\alpha^{-2}\b\omega^{-2}\psi(h_{1})[v\tl h_2]_{(1)}=v_{(0)}\tl \psi(h_{1})\o\b(v_{(1)})A\alpha^{-1}\omega^{-1}(h_{2}),
\end{align*} 
replacing $v$ by $T_{V}(v)$ and $h$ by $TA^{-1}\a^{2}\om\psi^{-1}(h_{2})$.

The coherence condition for $\xi$ is commutativity of the diagram
$$\xymatrix{
  M\o_{H} N\o_{H} P \ar[d]_{\mathrm{id}\o \xi} \ar[r]^{\xi\o \mathrm{id}}  & (M\B_{H} N)\o_{H} P \ar[d]^{\xi}  \\
  M\o_{H} (N\B_{H} P)  \ar[r]_{\xi}   & M\B_{H} N\B_{H} P.             }
                $$

This completes the proof.
\end{proof}
In what follows, we will construct a braiding structure on the strict monoidal category $(~\!^{H}_{H}\mathfrak{M}^{H}_{H},\o_{H})$. This is an important result in this section. In the same manner, we can construct a braiding structure on the strict monoidal category $(~\!^{H}_{H}\mathfrak{M}^{H}_{H},\B_{H})$.
\begin{theorem}\label{T2}
Let $H$ be a BiHom-Hopf algebra with a bijective antipode $S$. Then $(~\!^{H}_{H}\mathfrak{M}^{H}_{H},\o_{H})$ is a strict braided monoidal category with a braiding
\begin{align*}
\widetilde{\si}:M\o_{H} N&\ra N\o_{H} M
\\m\o n&\m [\a^{-2}\om^{-2}(m_{[-1]1})\c \a_{N}^{-2}\om_{N}^{-1}(n_{(0)})]\c S\b^{-1}\psi^{-2}(n_{(1)1})
\\&\qquad\quad \o S\a^{-1}\om^{-2}(m_{[-1]2})\c[\b_{M}^{-2}\psi_{M}^{-1}(m_{[0]})\c \b^{-2}\psi^{-2}(n_{(1)2})]
\end{align*}
and the inverse
\begin{align*}
\widetilde{\si}^{-1} (n\o m)=[\a^{-2}\psi^{-2}(n_{(1)2})\c &\a^{-2}_{M}\psi^{-1}_{M}(m_{[0]})]\c S^{-1}\b^{-1}\psi^{-2}(m_{[-1]2})
\\&\o S^{-1}\a^{-1}\om^{-2}(n_{(1)1})\c [\b^{-2}_{N}\om^{-1}_{N}(n_{(0)})\c \b^{-2}\om^{-2}(m_{[-1]1})],
\end{align*}
for any $m\in M$ and $n\in N$.
\end{theorem}
\begin{proof}
By Theorem \ref{T1}, it remains to check that the braiding $\widetilde{\si}$ induced in $\!^{H}_{H}\mathfrak{M}^{H}_{H}$ via the monoidal equivalent with $\mathcal{YD}^{H}_{H}(m,n,p,q)$ has the stated form.

We first check that the linear map $\widetilde{\si}$ is well defined. For any $h\in H$, $m\in M$ and $n\in N$, we have
\begin{align*}
&\widetilde{\si}(m\c h\o_{H} \b_{N}(n))
\\&=(\a^{-2}\om ^{-2}((m\c h)_{[-1]1})\c \a_{N}^{-2}\b_{N}\om_{N} ^{-1}(n_{(0)}))\c S\psi ^{-2}(n_{(1)1})
\\&\quad\o_{H} S\a^{-1}\om ^{-2}((m\c h)_{[-1]2})\c(\b_{M}^{-2}\psi_{M}^{-1}((m\c h)_{[0]})\c \b^{-1}\psi ^{-2}(n_{(1)2}))
\\&=(\a^{-2}\om ^{-2}(m_{[-1]1} h_{11})\c \a_{N}^{-2}\b_{N}\om_{N} ^{-1}(n_{(0)}))\c S\psi ^{-2}(n_{(1)1})
\\&\quad\o_{H} S\a^{-1}\om ^{-2}(m_{[-1]2} h_{12})\c(\b_{M}^{-2}\psi_{M}^{-1}(m_{[0]}\c h_{2})\c \b^{-1}\psi ^{-2}(n_{(1)2}))
\\&=(\a^{-1}\om ^{-2}(m_{[-1]1} )\c (\a^{-2}\om ^{-2}(h_{11})\c \a_{N}^{-2}\om_{N} ^{-1}(n_{(0)})))\c S\psi ^{-2}(n_{(1)1})
\\&\quad \o_{H} S\a^{-2}\b\om ^{-2}(h_{12})S\b^{-1}\om ^{-2}(m_{[-1]2})\c(\a_{M}\b_{M}^{-2}\psi_{M}^{-1}(m_{[0]})\c \b^{-2}\psi ^{-2}(\psi(h_{2}) n_{(1)2}))
\\&=(\a^{-1}\om ^{-2}(m_{[-1]1} )\c (\a^{-2}\om ^{-2}(h_{11})\c \a_{N}^{-2}\om_{N} ^{-1}(n_{(0)})))\c S\psi ^{-2}(n_{(1)1})
\\&\quad \o_{H} S\a^{-1}\b\om ^{-2}(h_{12})\c [S\b^{-1}\om ^{-2}(m_{[-1]2})\c\b_{M}^{-1}(\a_{M}\b_{M}^{-2}\psi_{M}^{-1}(m_{[0]})\c \b^{-2}\psi ^{-2}(\psi(h_{2}) n_{(1)2}))]
\\&=[\a_{N}^{-1}(\a^{-1}\om ^{-2}(m_{[-1]1} )\c \a_{N}^{-2}\om_{N} ^{-1}(\om^{-1}(h_{11})\c n_{(0)}))\c S\a^{-1}\psi ^{-2}(n_{(1)1})]
\c S\a^{-1}\b\om ^{-2}(h_{12})
\\&\quad \o_{H}  S\om ^{-2}(m_{[-1]2})\c(\a_{M}\b_{M}^{-2}\psi_{M}^{-1}(m_{[0]})\c \b^{-2}\psi ^{-2}(\psi(h_{2}) n_{(1)2}))
\\&=(\a^{-1}\om ^{-2}(m_{[-1]1} )\c \a_{N}^{-2}\om_{N} ^{-1}(\om^{-1}(h_{11})\c n_{(0)}))\c S\a^{-1}\psi ^{-2}(n_{(1)1}) S\a^{-1}\om ^{-2}(h_{12})
\\&\quad \o_{H}  S\om ^{-2}(m_{[-1]2})\c(\a_{M}\b_{M}^{-2}\psi_{M}^{-1}(m_{[0]})\c \b^{-2}\psi ^{-2}(\psi(h_{2}) n_{(1)2}))
\\&=(\a^{-1}\om ^{-2}(m_{[-1]1})\c \a_{N}^{-2}\om_{N} ^{-1}(h_{1}\c n_{(0)}))\c S\b^{-1}\psi ^{-2}(h_{21}n_{(1)1})
\\&\quad\o_{H} S\om ^{-2}(m_{[-1]2})\c(\a_{M}\b_{M}^{-2}\psi_{M}^{-1}(m_{[0]})\c \b^{-2}\psi ^{-2}(h_{22} n_{(1)2}))
\\&=(\a^{-1}\om ^{-2}(m_{[-1]1})\c \a_{N}^{-2}\om_{N} ^{-1}((h\c n)_{(0)}))\c S\b^{-1}\psi ^{-2}((h\c n)_{(1)1})
\\&\quad\o_{H} S\om ^{-2}(m_{[-1]2})\c(\a_{M}\b_{M}^{-2}\psi_{M}^{-1}(m_{[0]})\c \b^{-2}\psi ^{-2}((h\c n)_{(1)2}))
\\&=\widetilde{\si}(\a_{M}(m)\o_{H}h\c n)
\end{align*}

Next, it is sufficient to consider the case of BiHom-four-angle Hopf modules $M=H\o V$ and $N=H\o W$ with $V, W\in \mathcal{YD}^{H}_{H}(m,n,p,q)$. In this case the $\widetilde{\si}$ is defined by the commutative diagram
$$\xymatrix{
  (H\o V)\o_{H}(H\o W) \ar[d]_{\widetilde{\si}} \ar[r]^-{\vp}      &  (H\o V\o W)\ar[d]^{\mathrm{id}\o c_{V,W}}  \\
  (H\o W)\o_{H}(H\o V)  \ar[r]_-{\vp}               & H\o W\o V           }
  $$
where the $c_{V,W}$ denotes the braiding in Theorem \ref{T3.5}. For any $g, h\in H$, $v\in V$ and $w\in W$, we have
\begin{align*}
&\widetilde{\si}(g\o v\o_{H} h\o w)=\vp^{-1}\circ(\mathrm{id}\o c_{V,W})\circ\vp(g\o v\o_{H} h\o w)
\\&=\vp^{-1}\circ(\mathrm{id}\o c_{V,W})(T(g)T\om^{-1}(h_{1})\o T_{V}(v)\tl  TA^{-1}\a^{2}\om(h_{2})\o T_{W}\b_{W}(w))
\\&=\vp^{-1}(T(g)T\om^{-1}(h_{1})\o T_{W}\b_{W}\om_{W}^{-1}(w_{(0)})\o \a_{V}^{-1}(T_{V}(v)\tl  TA^{-1}\a^{2}\om(h_{2}))\tl TA^{-1}\a\b\om(w_{(1)}))
\\&=\a^{-1}(g)\a^{-1}\om^{-1}(h_{1})\o a_{W}^{-1}\b_{W}\om_{W}^{-1}(w_{(0)})\o_{H}1\o  \b_{V}^{-1}(v)\tl  A^{-1}\a\b^{-1}\om(h_{2}) A^{-1}\a\b^{-1}\om(w_{(1)})
\\&=\a^{-1}\om ^{-1} (g_{1})\a^{-1}\om ^{-1}(h_{1}) \o \a_{W}^{-1}\b_{W}\om_{W}^{-1}(w_{(0)})\o_{H}  S\a^{-1}\b^{-1}\om^{-1}\psi ^{-2}(h_{211}w_{(1)11})
\\&\quad[\b^{-1}[S\a^{-2}\om ^{-2}(g_{21})\b^{-2}\psi^{-2}(g_{22})]\a^{-1}\b^{-2}\om^{-1}\psi^{-2}(h_{212}w_{(1)12})]
\\&\quad \o \b_{V}^{-1}(v)\tl A^{-1}\a\b^{-1}\om \psi^{-1} (h_{22}w_{(1)2})
\\&=(\a^{-1}\om ^{-2}(g_{11})\a^{-1}\om ^{-1}(h_{1}) \o \a_{W}^{-1}\b_{W}\om_{W}^{-1} (w_{(0)}))\o_{H}  S\a^{-1}\b^{-1}\psi ^{-2}(h_{21}w_{(1)1}) \c \\&\quad[\b^{-1}[S\a^{-2}\om ^{-2}(g_{12})\b^{-2}\psi^{-1}(g_{2})]\a^{-1}\b^{-2}\om^{-1}\psi^{-2}(h_{221}w_{(1)21})
\\&\quad \o \b_{V}^{-2}(v)\tl A^{-1}\a\b^{-2}\om \psi^{-2}(h_{222}w_{(1)22})]
\\&=(\a^{-2}\om ^{-2}(g_{11})\a^{-2}\om ^{-1}(h_{1}) \o \a_{W}^{-2}\b_{W}\om_{W}^{-1}(w_{(0)}))\c S\a^{-1}\b^{-1}\psi ^{-2}(h_{21}w_{(1)1})
\\&\quad \o_{H} S\a^{-1}\om ^{-2}(g_{12})[\b^{-2}\psi^{-1}(g_{2})\a^{-1}\b^{-2}\om^{-1}\psi^{-2}(h_{221}w_{(1)21})]
\\&\quad\o \b_{V}^{-1}(v)\tl A^{-1}\a\b^{-1}\om \psi^{-2}(h_{222}w_{(1)22})
\\&=(\a^{-2}\om ^{-2}(g_{11})\a^{-2}\om ^{-1}(h_{1}) \o \a_{W}^{-2}\b_{W}\om_{W}^{-1} (w_{(0)}))\c S\a^{-1}\b^{-1}\psi ^{-2}(h_{21}w_{(1)1})
\\&\quad \o_{H} S\a^{-1}\om ^{-2}(g_{12})\c(\b^{-2}\psi^{-1}(g_{2})\a^{-1}\b^{-2}\om^{-1}\psi^{-2}(h_{221}w_{(1)21})
\\&\quad\o \b_{V}^{-2}(v)\tl A^{-1}\a\b^{-2}\om \psi^{-2}(h_{222}w_{(1)22}))
\\&=(\a^{-2}\om ^{-2}(g_{11})\c \a_{N}^{-2}\om_{N} ^{-1}(h_{1}\o w_{(0)}))\c S\a^{-1}\b^{-1}\psi ^{-2}(h_{21}w_{(1)1})
\\&\quad \o_{H} S\a^{-1}\om ^{-2}(g_{12})\c(\b_{M}^{-2}\psi_{M}^{-1}(g_{2}\o \psi_{V}(v))\c \a^{-1}\b^{-2}\psi ^{-2}(h_{22}w_{(1)2}))
\\&=(\a^{-2}\om ^{-2}((g\o v)_{[-1]1})\c \a_{N}^{-2}\om_{N} ^{-1}((h\o w)_{(0)}))\c S\b^{-1}\psi ^{-2}((h\o w)_{(1)1})
\\&\quad \o_{H} S\a^{-1}\om ^{-2}((g\o v)_{[-1]2})\c(\b_{M}^{-2}\psi_{M}^{-1}((g\o v)_{[0]})\c \b^{-2}\psi ^{-2}((h\o w)_{(1)2})).
\end{align*}

This completes the proof.
\end{proof}
As a dual of the Theorem \ref{T2}, define the following linear map, for any $m\B_{H} n\in M\B_{H}N$:
\begin{align*}
\widehat{\si}:M\B_{H} N&\ra N\B_{H} M
\\m\B_{H} n&\m \a^{-2}\om^{-2}(m_{(0)[-1]})S\a^{-2}\om^{-1}(n_{[-1]})\c \b^{-1}_{N}\psi^{-2}_{N}(n_{[0](0)})
\\&\qquad \quad\B_{H} \a^{-1}_{M}\om^{-2}_{M}(m_{(0)[0]})\c S\b^{-2}\psi^{-1}(m_{(1)})\b^{-2}\psi^{-2}(n_{[0](1)}).
\end{align*}
Note that the $\widehat{\si}$ is bijective with inverse 
$\widehat{\si}^{-1}(n\B_{H} m)=\a^{-2}_{M}\psi^{-1}_{M}(m_{[0](0)})\c S^{-1}\b^{-1}\psi^{-2}(m_{[-1]})$ $\a^{-2}\psi^{-2}(n_{(0)[-1]})
\B_{H} \b^{-2}\om^{-2}(m_{[0](1)})S^{-1}\a^{-1}\om^{-2}(n_{(1)})\c \b^{-2}_{N}\om^{-1}_{N}(n_{(0)[0]})$.

\begin{theorem}
Let $H$ be a BiHom-Hopf algebra with a bijective antipode $S$. Then $(~\!^{H}_{H}\mathfrak{M}^{H}_{H},\B_{H})$ is a braided monoidal category with a braiding $\widehat{\si}$.
\end{theorem}
The proof is similar to Theorem \ref{T2}.

\section*{Acknowledgements}The authors are deeply indebted to the referee for the referee’s very  useful suggestions and some improvements to this paper. This work was partially supported by the Scientific Research Foundation of Nanjing Institute of Technology (No. YKJ202219) and the Natural Science Foundation of Jiangsu Higher Education Institutions of China (No. 22KJB110019).

\end{document}